\documentclass[14pt]{amsart}
\usepackage[margin=3cm]{geometry}

\usepackage{microtype} 

\usepackage{amsthm,amsmath,amssymb,amsfonts}
\usepackage[alphabetic]{amsrefs}
\usepackage[skip=4pt plus1pt, indent=10pt]{parskip} 
\usepackage{enumerate}
\usepackage{xcolor}
\usepackage{tikz}
\usepackage[colorlinks,bookmarks, anchorcolor=blue, citecolor=blue, linkcolor=blue,urlcolor=blue, bookmarksopenlevel=1]{hyperref}
\usepackage{tikz-cd}

\newtheorem{theorem}{Theorem}[section]
\newtheorem{proposition}[theorem]{Proposition}
\newtheorem{lemma}[theorem]{Lemma}
\newtheorem{corollary}[theorem]{Corollary}

\theoremstyle{definition}

\newtheorem{remark}[theorem]{Remark}

\begin{document}\baselineskip=15pt

\begin{center}
\title[Uniruledness of modular varieties]{A simple criterion for the uniruledness of an orthogonal modular variety}

\author{Ignacio Barros}
\address{\parbox{0.9\textwidth}{
Department of Mathematics\\[1pt]
Universiteit Antwerpen\\[1pt]
Middelheimlaan 1, 2020 Antwerpen, Belgium
\vspace{1mm}}}
\email{{ignacio.barros@uantwerpen.be}}


\maketitle
\end{center}

\begin{abstract}
We exhibit a simple uniruledness criterion for general orthogonal modular varieties in terms of invariants of the corresponding lattice. As an application, we obtain the uniruledness of almost all Nikulin--Vinberg moduli spaces parameterizing projective K3 surfaces of Picard number at least $3$ and fixed finite automorphism group.
\end{abstract}

\setcounter{tocdepth}{1} 

\section{Introduction}

A central problem in geometry bringing together automorphic forms, birational geometry, and moduli theory is the study of the birational geometry of arithmetic quotients of Hermitian symmetric domains. For example, the moduli space of principally polarized abelian varieties $\mathcal{A}_g$ is an arithmetic quotient of the Siegel upper-half space $\mathbb{H}_g$. The Kodaira classification of $\mathcal{A}_g$ was an influential problem studied by Tai, Freitag, Mumford, and Donagi \cites{Tai82, Fre83, Mum83, Don84} leaving open only the case $g=6$, see \cite{DSMS21}. Another prominent example is the one of orthogonal modular varieties. Let $G={\rm{O}}(V)$, where $(V,q)$ is a real quadratic space of signature $(b,2)$, and $K$ a maximal compact subgroup. Arithmetic quotients of $G\big/K$ are natural instances of moduli spaces of projective K3 surfaces, cubic fourfolds, and hyperk\"{a}hler varieties, all fundamental objects of geometry. The study of the birational geometry of moduli spaces of K3 surfaces, in particular the Kodaira classification was initiated by Mukai, Kond\={o}, and Gritsenko \cites{Muk88, Muk92, Kon93, Gri94, Muk96, Kon99}, and studied further (also beyond the $2$-dimensional case) by many authors, cf. \cites{Muk06, GHS07, Muk10, GHS10, GHS11, FV14, TVA19, FV21, BBBF23}. A first systematic study of the Kodaira dimension of orthogonal modular varieties was carried out by Ma \cite{Ma18} building on \cite{GHS07b}. Ma's result confirms the general expectation that orthogonal modular varieties are almost always of general type, and those of negative Kodaira dimension should appear only in low dimension and for special arithmetic groups $\Gamma\subset G$. The main goal of this paper is to complement \cite{Ma18} by establishing a general criterion for an orthogonal modular variety to have negative Kodaira dimension.

The first study of uniruledness in the general context was carried out in \cites{Gri10, GH14}, where Gritsenko and Hulek show that an orthogonal modular variety $\mathcal{D}_L\big/\Gamma$ is uniruled if it admits a reflective modular form of large weight and controlled vanishing, see \cite{GH14}*{Theorems 1.3 and 2.1} for details. Yet the problem of determining the existence of such modular forms is rather delicate and in general open. See \cite{Wan24} for recent progress on the subject. 

Let $L$ be an even lattice of signature $(b,2)$ with quadratic form $q_L(\cdot)=\frac{\langle \cdot,\cdot\rangle}{2}$. We will assume $b>2$. The group $G={\rm{O}}(L_\mathbb{R})$ acts transitively on the Grassmannian ${\rm{Gr}}^-(2,L_\mathbb{R})$ of $2$-planes $\pi$ in $L_{\mathbb{R}}$ where the restriction $q_L\mid_\pi$ is negative-definite. The stabilizer of a point is a maximal compact subgroup isomorphic to $K={\rm{O}}(b)\times{\rm{O}}(2)$ and the symmetric domain ${\rm{Gr}}^-(2,L_\mathbb{R})\cong G\big/K$ can be seen as an analytic open in the quadric $\mathcal{Q}_L\subset \mathbb{P}\left(L_\mathbb{C}\right)$ defined by $q_L(Z)=0$. Concretely, the subset 
\[
\mathcal{D}_L\cup\overline{\mathcal{D}}_L=\left\{[Z]\in \mathcal{Q}_L\left|\langle Z,\overline{Z}\rangle<0\right.\right\}
\]
has two components exchanged by complex conjugation and $\mathcal{D}_L\cong {\rm{Gr}}^-(2,L_\mathbb{R})\cong G\big/K$. We denote by $A_L$ the discriminant group $L^{\vee}\big/L$ together with the corresponding $\mathbb{Q}\big/\mathbb{Z}$-valued quadratic form $q_L\mod \mathbb{Z}$. The {\textit{discriminant of $L$}} is defined by $D=\left|A_L\right|$ and the {\textit{level}} is the smallest positive integer $N$ such that $N\cdot q_L$ is integral on $L^{\vee}$, or equivalently, trivial on $A_L$. The {\it{stable orthogonal group}} is the subgroup $\widetilde{\rm{O}}^+\left(L\right)$ of $\rm{O}(L)$ fixing the component $\mathcal{D}_L$ and acting trivially on $A_L$. The quotient
\[
\mathcal{F}_L=\mathcal{D}_L\big/\widetilde{\rm{O}}^+\left(L\right)
\]
is a quasi-projective variety of dimension $b$. Our main result is the following:

\begin{theorem}
\label{thm:int:main}
Let $L$ be an even lattice of signature $(b,2)$, level $N$, and discriminant $D$ splitting off one copy of $U$. Let $k=\frac{1}{2}{\rm{rk}}(L)=\frac{b}{2}+1$ and $p$ prime. Then $\mathcal{F}_L$ is uniruled provided 
\begin{equation}
\label{eq:sec2:main_1}
4b<\frac{\left(2\pi\right)^{k}}{\sqrt{D}\cdot\Gamma\left(k\right)\cdot\zeta\left(\lfloor k\rfloor\right)}\cdot C(N,k),
\end{equation}
where 
\[
C(N,k)=\prod_{p\mid 2N}1-\frac{1}{p}\;\;\hbox{if $b$ is even, and }\;\;C(N,k)=\prod_{p\mid 2N}\frac{p-1}{p(1-p^{1-2k})}\;\;\hbox{if $b$ is odd.}
\]
Further, if $L$ splits off two copies of $U$, then $C(N,k)$ can be taken as
\[
C(N,k)=\prod_{p\mid 2N}1-\frac{1}{p^2}\;\;\hbox{if $b$ is even, and }\;\;C(N,k)=\prod_{p\mid 2N}\frac{p^2-1}{p^2(1-p^{1-2k})}\;\;\hbox{if $b$ is odd.}
\]
\end{theorem}

Here $\Gamma(\cdot)$ stands for the Gamma function and $\zeta(\cdot)$ for the Riemann zeta function. Note that for any arithmetic group $\Gamma\subset {\rm{O}}^+\left(L\right)$, if $\Gamma$ contains $\widetilde{\rm{O}}^+\left(L\right)$, then $\mathcal{F}_L$ dominates $\mathcal{D}_L\big/\Gamma$ and the uniruledness of the former implies the uniruledness of the latter. For convenience of the reader, in Table \ref{sec3:table:r(k)} we list the first values to the third decimal place of the quotient $\frac{(2\pi)^k}{\Gamma(k)\cdot\zeta(\lfloor k\rfloor)}$.

The proof of the theorem relies on three main ingredients: Miyaoka--Mori's bend-and-break \cite{MM86}, Kudla's volume formula for special divisors \cite{Kud03}, and Bruinier--Kuss formulas \cite{BK01} for the Fourier coefficients of vector-valued Eisenstein series. The idea of using \cite{Kud03} to show the negativity of $K_{\overline{\mathcal{F}}_{2d}}$ and obtain new cases of moduli spaces $\mathcal{F}_{2d}$ of primitively polarized K3 surfaces of degree $2d$ of negative Kodaira dimention was first pointed out by Peterson in \cite{Pet15}*{Section 4.7}\footnote{In \cite{Pet15}*{Theorem 4.7.3} it is claimed that $\kappa\left(\mathcal{F}_{2d}\right)<0$ for $d\leq 15$, yet we think there is a minor calculation mistake coming from a factor of $2$ missing in the denominator of the intersection $\lambda^{18}\cdot B$. According to our calculations $\lambda^{18}\cdot K_{\overline{\mathcal{F}}_{2d}}$ is negative for $d\leq 3$ and positive for $d\geq 4$. This gives uniruledness in cases where the unirationality of $\mathcal{F}_{2d}$ is well-known.}. Peterson's strategy was further used in \cite{BBFW24}*{Section 4} to establish a uniruledness criterion in terms of Fourier coefficients of Eisenstein series and applied to study the uniruledness of various moduli spaces of polarized hyperk\"{a}hler manifolds. Here we focus on the general case and use the bounds obtained in \cite{BM19} to get a simple criterion. 

In the last section of the paper, we discuss some examples and applications. As an application, we obtain the uniruledness of Nikulin--Vinberg moduli spaces parameterizing K3 surfaces of Picard rank $\rho\geq 3$ with a fixed finite automorphism group. In \cites{Nik79,Nik81,Nik84,Vin07}, Nikulin and Vinberg classified all possible N\'{e}ron--Severi lattices of projective K3 surfaces of Picard number $\rho\geq 3$ with finite automorphism group and Kond\={o} \cite{Kon89} listed the possible automorphism groups. Further, in \cite{Rou22} the list of finite groups corresponding to each one of the $118$ families was completed. This leads to $118$ moduli spaces parameterizing lattice-polarized K3 surfaces \cite{Dol96} of Picard rank $\rho\geq 3$ whose very general element has a fixed finite automorphism group. See \cite{Rou22}*{Section 16} for the complete list. Several of these moduli spaces were classically known to be (uni)rational and the unirationality of many more was established in \cite{Rou22} leaving open $55$ cases. 

\begin{theorem}
\label{sec1:thm:Main2}
All $118$ Nikulin--Vinberg moduli spaces parameterizing projective K3 surfaces of Picard rank $\rho\geq 3$ with fixed finite automorphism group are uniruled with $12$ possible exceptions. See Table \ref{sec4:table:thm4.1}. 
\end{theorem}

Recently in \cite{CMW24} the authors established an explicit rational parametrization of various moduli spaces of lattice-polarized K3 surfaces of large Picard number. In particular the last two cases in Table \ref{sec4:table:thm4.1} are rational.

The paper is organized as follows. In Section \ref{sec2} we recall some preliminaries on vector-valued modular forms and the relation between intersection numbers and coefficients of Eisenstein series. In Section \ref{sec3} we prove Theorem \ref{thm:int:main} and finally in Section \ref{sec4} we discuss applications and prove Theorem \ref{sec1:thm:Main2}.

\subsection*{Acknowledgments}
I would like to thank D. Agostini, P. Beri, J. H. Bruinier, L. Flapan, P. Kiefer, A. L. Knutsen, and B. Williams for helpful comments, discussions, and email correspondences. I would also like to thank Pedro Montero for pointing us to relevant literature on the Nikulin--Vinberg moduli spaces. The author was supported by Fonds voor Wetenschappelijk Onderzoek – Vlaanderen (FWO, Research Foundation – Flanders) -- G0D9323N and Deutsche Forschungsgemeinschaft (DFG, German Research Foundation) -- – Project-ID 491392403 – TRR 358.

\section{Vector valued Eisenstein series and intersection numbers}
\label{sec2}

In this section we gather some background material on Heegner divisors, vector-valued modular forms with respect to the Weil representation, and the connection with intersection numbers.\\

Let $L$ is an even lattice of signature $(b,2)$ splitting off one copy of the hyperbolic plane $U$ and $\mathbb{C}\left[A_L\right]$ the group algebra of the discriminant group $A_L$ with standard $\mathbb{C}$-basis $\left\{\mathfrak{e}_\mu\left|\mu\in A_{L}\right.\right\}$. Attached to such data there is a distinguished representation of the metaplectic double cover of ${\rm{SL}}_2(\mathbb{Z})$
\[
\rho_L:\widetilde{{\rm{SL}}_2}(\mathbb{Z})\longrightarrow {\rm{GL}}\left(\mathbb{C}\left[A_L\right]\right)
\]
called the {\it{Weil representation}}, see \cite{Bor98}. This gives rise to modular forms with values in $\mathbb{C}\left[A_L\right]$. A {\it{modular form}} $f\in {\rm{Mod}}_{k,L}$ of weight $k\in\frac{1}{2}\mathbb{Z}$ with respect to $\rho_L$ is an holomorphic function 
\[
f:\mathbb{H}\longrightarrow\mathbb{C}\left[A_L\right]
\]
holomorphic at the cusp $i\infty$ and such that for all $g=(A,\phi)\in \widetilde{{\rm{SL}}_2}(\mathbb{Z})$ and $\tau\in\mathbb{H}$, the function $f$ transforms as
\[
f(g\tau)=\phi(\tau)^{2k}\rho_L(g)f(\tau).
\]
Every such modular form admits a Fourier expansion centered at the cusp at infinity of the form
\[
f(\tau)=\sum_{\substack{\mu\in D(L)\\n\in \frac{1}{N}\mathbb{Z}_{\geq0}}}a_{n,\mu}q^n\mathfrak{e}_{\mu},
\]
where $q=e^{2\pi i \tau}$. They form a finite-dimensional vector space that admits a basis of modular forms with Fourier coefficients in $\mathbb{Q}$, giving ${\rm{Mod}}_{k,L}$ the structure of a $\mathbb{Q}$-vector space, see \cite{Bor99}*{Lemma 4.2} and \cite{McG03}*{Theorem 5.6}. Recall that $f$ is called a cusp form if $a_{0,\mu}=0$ for all isotropic elements $\mu\in A_L$. Cusp forms define a sub $\mathbb{Q}$-vector space denoted ${\rm{S}}_{k,L}$. Following \cite{BK01}, assume $k>2$ and $2k-2+b\equiv 0\mod4$. If $\widetilde{\Gamma}_{\infty}\subset\widetilde{\rm{SL}}_2\left(\mathbb{Z}\right)$ is the stabilizer of the cusp $i\infty$, then 
\[
E_{k,L}=\frac{1}{4}\sum_{g\in\widetilde{\Gamma}_{\infty}\big\backslash \widetilde{\rm{SL}}_2\left(\mathbb{Z}\right)}\phi(\tau)^{-2k}\rho_L(g)^{-1}\mathfrak{e}_0=\mathfrak{e}_0+\sum_{\substack{n,\mu\\n>0}}e_{n,\mu}q^n\mathfrak{e}_\mu
\]
is in ${\rm{Mod}}_{k,L}$. The Fourier coefficients $e_{n,\mu}$ of $E_{k,L}$ were computed explicitly in \cite{BK01}, in particular, they are all rational numbers, see \cite{BK01}*{Corollary 8}. One defines the $\mathbb{Q}$-vector space
\[
{\rm{Mod}}_{k,L}^\circ:=\mathbb{Q}E_{k,L}\oplus {\rm{S}}_{k,L}.
\]
Consider the Fourier coefficient extraction functional $c_{n,\mu}\in\left({\rm{Mod}}_{k,L}^\circ\right)^{\vee}$ given by 
\[
c_{n,\mu}:{\rm{Mod}}_{k,L}^\circ\longrightarrow \mathbb{Q},\;\;\;f(\tau)=\sum a_{n,\mu}q^n\mathfrak{e}_\mu\mapsto a_{n,\mu}.
\]
Important for us will be the following Fourier coefficient of $E_{k,L}$:

\begin{lemma}[See Theorem 11 in \cite{BK01}]
\label{lemma:sec2:(1,0)-coef}
If $k=\frac{b}{2}+1>2$, the $(1,0)$-Fourier coefficient of $E_{k,L}$ is given by 
\[
c_{1,0}\left(E_{k,L}\right)=-\frac{(2\pi)^k}{\sqrt{\left|A_L\right|}\cdot \Gamma(k)\cdot L(k,\chi_{4D})}\cdot\left(\prod_{\substack{p\hbox{ prime}\\p\mid 2N}}\frac{N_{1,0}(p^{w_p})}{p^{(2k-1)w_p}}\right)
\]
if $b$ is even, and by
\[
c_{1,0}\left(E_{k,L}\right)=-\frac{(2\pi)^k\cdot L(k-\frac{1}{2},\chi_{D'})}{\sqrt{\left|A_L\right|}\cdot \Gamma(k)\cdot \zeta(2k-1)}\cdot\left(\prod_{\substack{p\hbox{ prime}\\p\mid 2N}}\frac{N_{1,0}(p^{w_p})}{(1-p^{1-2k})p^{(2k-1)w_p}}\right)
\]
if $b$ is odd. Where in the formulas:
\begin{itemize}
\item[-] $\Gamma(\cdot)$ is the Gamma function, $\zeta(\cdot)$ the Riemann zeta function, $L(\cdot,\chi)$ the Dirichlet L-function with character $\chi$, and $\chi_{4D}$ (resp. $\chi_{D'}$) stands for the Dirichlet character modulo $|4D|$ (resp. $|D'|$) with $D=(-1)^k\det(L)$ (resp. $D'=2(-1)^{k+\frac{1}{2}}\cdot\det(L)$).
\item[-] The value $N_{1,0}(a)$ is given by 
\[
N_{1,0}(a)=\left|\left\{v\in L\big/aL\left|\;q_L(v)\equiv 1\mod a\right.\right\}\right|.
\]
\item[-] The value $w_p$ is given by 
\[
w_p=\begin{cases}
3&\hbox{ if }p=2\\
1&\hbox{ otherwise.}
\end{cases}
\]

\end{itemize}
\end{lemma}

For a lattice $L$ as above, the locally symmetric domain $\left\{\left[Z\right]\in\mathbb{P}\left(L_{\mathbb{C}}\right)\left|\langle Z,Z\rangle=0, \langle Z,\overline{Z}\rangle<0\right.\right\}$ has two components $\mathcal{D}_L\cup\overline{\mathcal{D}}_L$ exchanged by complex conjugation. One denotes by ${\rm{O}}^+\left(L\right)$ the subgroup of ${\rm{O}}\left(L\right)$ fixing $\mathcal{D}_L$, and by $\widetilde{\rm{O}}^+\left(L\right)$ the kernel of the discriminant map ${\rm{O}}^+\left(L\right)\longrightarrow {\rm{O}}\left(A_L\right)$. The {\textit{stable orthogonal modular variety}} $\mathcal{F}_L$ is defined as the quotient
\[
\mathcal{F}_L:=\mathcal{D}_L\big/\widetilde{\rm{O}}^+\left(L\right).
\]
It is a normal quasi-projective variety of dimension $b$, see \cite{BB66}. Let $v\in L^\vee$, we denote by $D_v$ the hyperplane $v^{\perp}\cap \mathcal{D}_L$ defined by $v$. For any coset $\mu+L\in A_L$ and $n\in\mathbb{Q}_{>0}$ the image of the cycle 
\[
\sum_{\substack{v\in \mu+L\\q_L(v)=n}}D_v
\]
via the $\widetilde{\rm{O}}^+\left(L\right)$-quotient is a $\mathbb{Q}$-Cartier divisor $H_{n,\mu}$ on $\mathcal{F}_L$ called a {\it{Heegner divisor}}. When there is no element $v$ in the coset $\mu+L$ with $q_L(v)=n$ one declares $H_{n,\mu}=0$. 

\begin{remark}
\label{sec2:rmk:2H}
Heegner divisors $H_{n,\mu}$ are in general neither reduced nor irreducible. Further, every component has multiplicity two if $2\mu=0$ in $A_L$ and multiplicity one otherwise.
\end{remark}

Further, $H_{0,0}$ is identified with $\mathcal{L}^{\vee}$, where $\mathcal{L}$ is the Hodge class, that is, the descent of $\mathcal{O}(-1)$ on $\mathcal{D}_L$ to $\mathcal{F}_L$. The fundamental result \cite{Bor99}*{Theorem 4.5} refined in \cite{McG03}, and building on earlier work of Kudla--Millson \cite{KM90} (see also \cite{Bru02}) is that the generating series
\begin{equation}
\label{eq:sec2:gen_ser}
H_{0,0}\mathfrak{e}_0+\sum_{\substack{n>0\\\mu\in A_L}}H_{n,\mu}q^n\mathfrak{e}_\mu\in{\rm{Pic}}_\mathbb{Q}\left(\mathcal{F}_L\right)[\![q^{\frac{1}{N}}]\!]\otimes\mathbb{C}\left[A_L\right]
\end{equation}
is actually an element in the finite-dimensional $\mathbb{Q}$-vector space ${\rm{Pic}}_\mathbb{Q}\left(\mathcal{F}_L\right)\otimes{\rm{Mod}}_{k,L}^\circ$. In particular, for any $\mathbb{Q}$-linear map
\[
\ell:{\rm{Pic}}_\mathbb{Q}\left(\mathcal{F}_L\right)\longrightarrow\mathbb{Q}
\]
the series $\sum \ell\left(H_{n,\mu}\right)q^n\mathfrak{e}_\mu$ is the Fourier expansion of a modular form in ${\rm{Mod}}_{k,L}^\circ$. This is particularly explicit when $\ell$ is given by the Baily--Borel degree map. Recall that the Baily--Borel model \cite{BB66} for $\mathcal{F}_L$ is the normal projective variety given by 
\[
\overline{\mathcal{F}}_L^{BB}:={\rm{Proj}}\left(\bigoplus_{s\geq0}H^0\left(\mathcal{F}_L,\mathcal{L}^{\otimes s}\right)\right)\subset \mathbb{P}^M.
\]
The boundary $\overline{\mathcal{F}}_L^{BB}\setminus\mathcal{F}_L$ is one dimensional whose components correspond to ${\rm{P}}\widetilde{\rm{O}}^+\left(L\right)$-orbits of rank two isotropic sublattices. In particular, if $b\geq 3$, then ${\rm{CH}}^1\left(\mathcal{F}_L\right)_{\mathbb{Q}}\cong{\rm{CH}}^1\left(\overline{\mathcal{F}}_L^{BB}\right)_{\mathbb{Q}}$ and there is a natural $\mathbb{Q}$-linear map
\[
{\rm{deg}}:{\rm{Pic}}_\mathbb{Q}\left(\mathcal{F}_L\right)\longrightarrow {\rm{CH}}^1\left(\overline{\mathcal{F}}_L^{BB}\right)_{\mathbb{Q}}\longrightarrow\mathbb{Q}
\]
sending $H_{n,\mu}$ to the degree of the closure $\overline{H}_{n,\mu}^{BB}$ with respect to $\mathcal{L}$ on the Baily--Borel model. Alternatively \cite{Kud03}, if $\lambda=c_1\left(\mathcal{L}\right)$ is the first Chern class of the Hodge bundle, then
\[
{\rm{deg}}\left(H_{n,\mu}\right)=\int_{H_{n,\mu}}\lambda^{b-1}.
\]
Let ${\rm{vol}}\left(\mathcal{F}_{L}\right)={\rm{deg}}\left(\lambda\right)=\int_{\mathcal{F}_L}\lambda^n$. Applying ${\rm{deg}}:{\rm{Pic}}_{\mathbb{Q}}\left(\mathcal{F}_L\right)\longrightarrow \mathbb{Q}$ to the generating series \eqref{eq:sec2:gen_ser} one obtains \cite{Kud03}*{Theorem I}:
\begin{equation}
\label{eq:sec2:Kud}
-{\rm{vol}}\left(\mathcal{F}_L\right)\cdot E_{\frac{b}{2}+1,L}=\sum_{n,\mu}{\rm{deg}}\left(H_{n,\mu}\right)q^n\mathfrak{e}_\mu.
\end{equation}

\section{Canonical class and uniruledness}
\label{sec3}

Let $L$ be an even lattice of signature $(b,2)$ with $b\geq3$. The canonical bundle on the analytic manifold $\mathcal{D}_L\subset\mathcal{Q}_L$ is given by $\mathcal{O}(-b)$, and the divisorial locus of the ramification of the modular projection
$\pi:\mathcal{D}_L\longrightarrow\mathcal{F}_L$ consist on hyperplanes $D_v$ fixed by reflections (up to sign) in $\widetilde{\rm{O}}^+\left(L\right)$, see \cite{GHS07}*{Corollary 2.13}. In particular, the divisorial ramification is of index two and by Riemann--Hurwitz
\begin{equation}
\label{eq:sec3:K}
K_{\mathcal{F}_L}=b\cdot\lambda-\frac{1}{2}{\rm{Br}}\left(\pi\right),
\end{equation}
where ${\rm{Br}}(\pi)$ is the reduced branch divisor of $\pi$. The following argument already appears in a different context in \cite{BBFW24}. We add it here for completeness.

\begin{proposition}[See \cite{BBFW24}*{Section 4}]
\label{sec3:prop:deg_neg}
If ${\rm{deg}}\left(K_{\mathcal{F}_L}\right)<0$, then $\mathcal{F}_L$ is uniruled.
\end{proposition}

\begin{proof}
Let $\eta:\overline{\mathcal{F}}_L^{tor}\longrightarrow \overline{\mathcal{F}}_L^{BB}$ be a toroidal compactification of $\mathcal{F}_L$, and $\delta$ the boundary divisor. Recall that $\overline{\mathcal{F}}_L^{tor}$ has at worst finite quotient singularities, see \cite{AMRT10}. Then
\[
K_{\overline{\mathcal{F}}_L^{tor}}=b\cdot\lambda-\frac{1}{2}{\rm{Br}}\left(\pi\right)-a\delta \;\;\;\hbox{in}\;\;{\rm{Pic}}_{\mathbb{Q}}\left(\overline{\mathcal{F}}_L^{tor}\right),
\]
where $a\in\mathbb{Q}$\footnote{From the construction of toroidal compactifications \cite{AMRT10} one actually has $a=1$ for an appropriate choice of toroidal compactification, cf. \cite{Ma18}*{Section 1.1}.}. Since the boundary of $\overline{\mathcal{F}}_L^{BB}$ is one-dimensional and $b\geq3$, one can choose a representative of the curve class $\left(\eta^*\lambda\right)^{b-1}\in{\rm{NS}}_1\left(\overline{\mathcal{F}}_L^{tor}\right)$ disjoint from the boundary $\delta$. The curve class $\left(\eta^*\lambda\right)^{b-1}$ is nef (is a covering curve) and for any $D\in {\rm{Pic}}_{\mathbb{Q}}\left(\mathcal{F}_L\right)$ one has
\begin{equation}
\label{sec3:eq:deg(D)}
{\rm{deg}}\left(D\right)=\left(\eta^*\lambda\right)^{b-1}\cdot \overline{D}^{tor},
\end{equation}
where $\overline{D}^{tor}\in {\rm{Pic}}_{\mathbb{Q}}\left(\overline{\mathcal{F}}_L^{tor}\right)$ is any divisor class on $\overline{\mathcal{F}}_L^{tor}$ restricting to $D$ on $\mathcal{F}_L$. In particular,
\begin{equation}
\label{sec3:eq:degK}
{\rm{deg}}\left(K_{\mathcal{F}_L}\right)=\left(\eta^*\lambda\right)^{b-1}\cdot K_{\overline{\mathcal{F}}_L^{tor}}.
\end{equation}
Let $\varepsilon:X_L\longrightarrow\overline{\mathcal{F}}_L^{tor}$ be a resolution of singularities and consider the nef curve class $\gamma=\varepsilon^*\left(\eta^*\lambda\right)^{b-1}\in {\rm{NS}}_1\left(X_L\right)$. The projection formula together with \eqref{sec3:eq:degK} give us
\[
\gamma\cdot K_{X_L}={\rm{deg}}\left(K_{\mathcal{F}_L}\right)<0.
\]
In particular, $K_{X_L}$ is not pseudo-effective and $X_L$ is uniruled, see \cite{MM86}*{Theorem 1} (cf. \cite{BDPP13}).

\end{proof}

The reflection $\sigma_\rho\in {\rm{O}}(L)$ with respect to $\rho\in L_{\mathbb{Q}}$ is defined by 
\[
\sigma_\rho:v\mapsto v-2\frac{\langle v,\rho\rangle}{\langle\rho,\rho\rangle}\rho
\]
and the fixed locus is precisely the hyperplane $D_\rho$. The modular projection $\pi:\mathcal{D}_L\longrightarrow\mathcal{F}_L$ is simply ramified along the union of all $D_\rho$ where $\sigma_\rho$ or $-\sigma_\rho$ is in $\widetilde{\rm{O}}^+\left(L\right)$. Roots of $L$ always induce reflections \cite{GHS07}*{Proposition 3.1}. In particular $\pi$ is ramified over $H_{1,0}$, compare with \cite{BBFW24}*{Remark 4.1}. 

\begin{proposition}[See \cite{BBFW24}*{Proposition 1.2}]
\label{sec3:prop:uniruled}
Let $k=\frac{b}{2}+1$. Then the modular variety $\mathcal{F}_L$ is uniruled if
\[
4b<\left|c_{1,0}\left(E_{k,L}\right)\right|.
\]
\end{proposition}

\begin{proof}
Since $H_{1,0}$ has multiplicity two, and $\pi:\mathcal{D}_L\longrightarrow\mathcal{F}_L$ is simply ramified over $H_{1,0}$, one has $\frac{1}{4}H_{1,0}\leq\frac{1}{2}{\rm{Br}}\left(\pi\right)$. Further, since $\left(\eta^*\lambda\right)^{b-1}$ is nef in $\overline{\mathcal{F}}_L^{tor}$, cf. Equation \eqref{sec3:eq:deg(D)}, one has 
\[
{\rm{deg}}\left(\frac{1}{4}H_{1,0}\right)\leq{\rm{deg}}\left(\frac{1}{2}{\rm{Br}}(\pi)\right)\;\;\hbox{and}\;\; {\rm{deg}}\left(K_{\mathcal{F}_L}\right)\leq b\cdot{\rm{vol}}\left(\mathcal{F}_L\right)-\frac{1}{4}{\rm{deg}}\left(H_{1,0}\right).
\]
From Equation \eqref{eq:sec2:Kud}, one has 
\[
\frac{{\rm{deg}}\left(K_{\mathcal{F}_L}\right)}{{\rm{vol}}\left(\mathcal{F}_L\right)}\leq b-\frac{1}{4}\cdot\left|c_{1,0}\left(E_{k,L}\right)\right| 
\]
and one concludes by Proposition \ref{sec3:prop:deg_neg}.
\end{proof}

Our main theorem follows from Lemma \ref{lemma:sec2:(1,0)-coef} and Proposition \ref{sec3:prop:uniruled}. 

\begin{theorem}[= Theorem \ref{thm:int:main}]
\label{thm:sec2:main}
Let $L$ be an even lattice of signature $(b,2)$, level $N$, and discriminant $D$ splitting off one copy of $U$. Let $k=\frac{1}{2}{\rm{rk}}(L)=\frac{b}{2}+1$ and $p$ prime. Then $\mathcal{F}_L$ is uniruled provided 
\begin{equation}
\label{eq:sec2:main}
4b<\frac{\left(2\pi\right)^{k}}{\sqrt{D}\cdot\Gamma\left(k\right)\cdot\zeta\left(\lfloor k\rfloor\right)}\cdot C(N,k),
\end{equation}
where $C(N,k)$ is as in Theorem \ref{thm:int:main}.
\end{theorem}

\begin{proof}
For any real Dirichlet character $\chi$ and $s>1$, 
\begin{equation}
\label{sec3:eq1:thm}
\frac{\zeta(2s)}{\zeta(s)}\leq L(s,\chi)\leq \zeta(s),
\end{equation}
see \cite{OS19}*{Lemma 3.9}. Further, by \cite{BM19}*{Lemmas 2.3 and 2.6}
\begin{equation}
\label{sec3:eq2:thm}
\prod_{p\mid 2N}1-\frac{1}{p}\leq\prod_{p\mid 2N}\frac{N_{1,0}(p^{w_p})}{p^{(2k-1)w_p}}\;\;\;\hbox{and}\;\;\;\prod_{p\mid 2N}1-\frac{1}{p^2}\leq\prod_{p\mid 2N}\frac{N_{1,0}(p^{w_p})}{p^{(2k-1)w_p}}
\end{equation}
if $L$ splits off two copies of $U$. Then Lemma \ref{lemma:sec2:(1,0)-coef} and equations \eqref{sec3:eq1:thm} and \eqref{sec3:eq2:thm} yield 
\begin{equation}
\label{sec3:eq:c10(E)}
\frac{\left(2\pi\right)^{k}}{\sqrt{D}\cdot\Gamma\left(k\right)\cdot\zeta\left(\lfloor k\rfloor\right)}\cdot C(N,k)\leq\left|c_{1,0}\left(E_{k,L}\right)\right|,
\end{equation}
and Proposition \ref{sec3:prop:uniruled} gives us the theorem. 
\end{proof}

\begin{table}[hbt!]
\caption{Value of $r(k)=\frac{(2\pi)^k}{\Gamma(k)\cdot\zeta(\lfloor k\rfloor)}$ to the third decimal place for $3\leq b\leq 20$.}
\label{sec3:table:r(k)}
\begin{tabular}{|l|l|l|l|l|l|l|l|l|l|}
\hline
      & $b=3$     & $b=4$     & $b=5$     & $b=6$     & $b=7$ & $b=8$ & $b=9$ & $b=10$ & $b=11$  \\ \hline
$r(k)$ & $45.254$ & $103.177$ & $155.642$ & $240.000$ & $310.318$ & $393.495$ & $452.255$ & $504.000$ & $526.601$ \\ \hline
      & $b=12$     & $b=13$     & $b=14$     & $b=15$     & $b=16$ & $b=17$ & $b=18$ & $b=19$ & $b=20$  \\ \hline
$r(k)$ & $532.495$ & $513.576$ & $480.000$ & $432.083$ & $377.769$ & $320.054$ & $264.000$ & $211.894$ & $165.959$\\ \hline
\end{tabular}
\end{table}

\section{Examples and applications}
\label{sec4}

In this section, we briefly discuss some examples of geometric nature that follow from Theorem \ref{thm:int:main}. As a first application, we recover \cite{GH14}*{Theorem 3.2} together with some extra cases. Here the convention for the signature of $L$ is $(2,b)$, but the analysis is the same.

\begin{corollary}
The stable orthogonal modular variety
\[
\mathcal{F}_L=\mathcal{D}_L\big/\widetilde{\rm{O}}^+\left(L\right)
\]
is uniruled for $L$ as in Table \ref{sec4:table:thm2}.
\end{corollary}

\begin{table}[hbt!]
\caption{Examples of uniruled orthogonal modular varieties}
\label{sec4:table:thm2}
\begin{tabular}{|c|c|c|}
\hline
\hbox{Lattice}&$(D,N,k)$&\hbox{Uniruled}\\ \hline
$U^2\oplus A_{2n+1}^{\oplus s}$&$\left((2n+2)^s, 4(n+1),\frac{(2n+1)s}{2}+2\right)$& $\begin{array}{rcl}(n,s)&=&(0,1\leq s\leq 6), (1,1\leq s \leq 3),\\ &&(2,1), (2,2), (3,1)\end{array}  $\\ \hline
$U^2\oplus A_{2n}^{\oplus s}$&$\left((2n+1)^s, 2n+1,ns+2\right)$&$\begin{array}{rcl}(n,s)&=&(1,1\leq s\leq 3), (3\leq n\leq 6, 1)\\&&(2,1), (2,2)\end{array}$\\ \hline
$U^2\oplus D_{2n+1}^{\oplus s}$&$\left(4^s, 8,\frac{(2n+1)s}{2}+2\right)$&$\begin{array}{rcl}(n,s)&=&(3\leq n\leq 8,1),(2,1), (2,2)\end{array}$ \\ \hline
$U^2\oplus D_{2n}^{\oplus s}$&$\begin{array}{cl}\left(4^s, 4,ns+2\right)&\hbox{ if }n\equiv 1\mod 2\\\left(4^s, 2,ns+2\right)&\hbox{ if }n\equiv 0\mod 2\end{array}$&$\begin{array}{rcl}(n,s)&=& (2\leq n\leq 8, 1),(2,2), (3,2)\end{array}$ \\ \hline
$U^2\oplus E_8^{\oplus s}$&$\left(1,1,4s+2\right)$& $s=1,2$ \\ \hline
$U^2\oplus E_7^{\oplus s}$&$\left(2^s, 4, \frac{7s}{2}+2\right)$& $s=1,2$ \\ \hline
$U^2\oplus E_6^{\oplus s}$&$\left(3^s,3,3s+2\right)$& $s=1,2$ \\ \hline
$U^2\oplus A_1(-d)$&$\left(2d,4d,\frac{5}{2}\right)$& $1\leq d\leq4$ \\ \hline
$U^2\oplus E_8(-1)\oplus A_1(-d)$&$\left(2d,4d,\frac{13}{2}\right)$& $1\leq d\leq 38$ \\ \hline
\end{tabular}
\end{table}

\begin{remark}
Some of the varieties $\mathcal{F}_L$ in Table \ref{sec4:table:thm2} when $L$ splits off two copies of $U$ were known to be rational or unirational see \cite{Rou22}*{Table 16},  \cite{WW24}*{Theorem 5.4}, and \cite{CMW24}.
\end{remark}

In light of Proposition \ref{sec3:prop:deg_neg}, the uniruledness criterion in Theorem \ref{thm:sec2:main} can be improved if one takes into account not only the divisor $H_{1,0}$ (resp. $H_{-1,0}$ if the signature of $L$ is assumed to be $(2,b)$) but the full branch divisor of $\pi:\mathcal{D}_L\longrightarrow \mathcal{F}_L$. Further, the bound in Theorem \ref{thm:int:main} follows from the inequality \eqref{sec3:eq:c10(E)}, yet this is not always optimal and by calculating $\left|c_{1,0}(E_{k,L})\right|$ using the formulas in Lemma \ref{lemma:sec2:(1,0)-coef} one can often obtain more cases. 

A natural collection of examples are lattice-polarized K3 surfaces, see \cite{Dol96}. Let $L$ be a lattice of signature $(1,n)$. An $L$-polarized K3 surface is a pair $(X,j)$ where $j$ is a primitive embedding $j:L\hookrightarrow {\rm{Pic}}\left(X\right)$. For instance, when $L=E_8(-2)\oplus A_1(d)$, then $L$-polarized K3 surfaces are known as {\it{standard Nikulin surfaces}} of degree $2d$, see \cite{vGS07} for details. One can see them as K3 surfaces admitting a symplectic involution. Non-standard Nikulin surfaces appear only in degree congruent to $0\mod 4$. Both standard and non-standard Nikulin surfaces form $11$-dimensional moduli spaces that display a remarkable connection with moduli spaces of Prym curves analogous to Mukai's classical relation K3 surfaces of low degree and curves of low genus, see \cites{FV12, FV16, KLV21}. The moduli space of standard Nikulin surfaces $\mathcal{N}_{2d}^s$ is (uni)rational for $d\leq 7$, see \cites{FV12,FV16, Ver16} and the moduli space of non-standard Nikulin surfaces $\mathcal{N}_{2d}^{ns}$ is (uni)rational for $d\leq 10$ even, see \cite{KLV20}*{Theorem 1.2}. Further, both are of general type for $d$ large enough, see \cite{Ma18}*{Theorem 1.3}. Our methods do not give new uniruledness results beyond the cases for which there are (uni)rational parameterizations. 

\subsection{Moduli of K3 surfaces with finite automorphism group}
Let $G$ a finite group. A similar source of interesting moduli spaces of special K3 surfaces is those whose general element has a fixed automorphism group of finite order. Algebraic K3 surfaces of Picard number $\rho\geq3$ with finite automorphism group were classified by Nikulin \cites{Nik79, Nik81, Nik84} and Vinberg \cite{Vin07}, see also \cite{Kon89}, and more recently \cite{Rou22}*{Table 16}. Out of this classification, there are $118$ possible N\'{e}ron-Severi lattices $L$ and all of them admit a primitive embedding $L\hookrightarrow \Lambda_{K3}=U^{\oplus 3}\oplus E_8^{\oplus 2}$ that is unique up to ${\rm{O}}^+\left(\Lambda_{K3}\right)$. This leads to $118$ irreducible moduli spaces partially compactified by
\begin{equation}
\label{sec4:eq:M_L}
\mathcal{M}_L=\mathcal{D}_{L^{\perp}}\big/{\rm{O}}^+\left(\Lambda_{K3},L\right),
\end{equation}
where ${\rm{O}}^+\left(\Lambda_{K3},L\right)\subset{\rm{O}}^+\left(L^\perp\right)$ are restrictions of isometries on $\Lambda_{K3}$ fixing $L$, see \cite{Rou22}*{Proposition 2.12}. Out of the $118$ Nikulin--Vinberg moduli spaces, $63$ are known to be unirational, see \cite{Rou22} and the reference therein. Several of them parameterize K3 surfaces arising as minimal resolutions of anti-canonical divisors in weighted projective threefolds $\mathbb{P}(\underline{a})$. There are $95$ possible weights $\underline{a}$ where the resolution is a K3 leading to $95$ moduli spaces. Several of them satisfy the condition that the general K3 has finite automorphism group, falling into one of the $118$ Nikulin--Vinberg spaces. The possible weights were classified by Reid (unpublished), see also \cite{Yon90} and \cite{Bel02}. Note that they are all unirational \cite{Rou22}*{Corollary 2.9}, indeed if $d(\underline{a})$ is the degree of $-K_{\mathbb{P}(\underline{a})}$, then the corresponding moduli space $\mathcal{M}_L$ is dominated by the linear system $\left|\mathcal{O}_{\mathbb{P}(\underline{a})}\left(d(\underline{a})\right)\right|$. The $55$ Nikulin--Vinberg moduli spaces for which uniruledness/unirationality was not known are listed in Table \ref{sec4:table:thm4.1} in decreasing order of dimension. The recent result in \cite{CMW24}*{Tables 5 and 6} establishes the rationality of the last two examples.

\begin{theorem}[= Theorem \ref{sec1:thm:Main2}]
\label{sec4:thm:NVmoduli}
All $118$ Nikulin--Vinberg moduli spaces parameterizing projective K3 surfaces of Picard rank $\rho\geq 3$ with fixed finite automorphism group are uniruled with $12$ possible exceptions. See Table \ref{sec4:table:thm4.1}. 
\end{theorem}

\begin{proof}
Most of them follow immediately from Theorem \ref{thm:sec2:main}. All of them referring to the main theorem in the last column on Table \ref{sec4:table:thm4.1} are obtained in the same way. As an example, consider the case $L=U\oplus A_1^{\oplus 3}$, this is example $(12)$ in the Table. Then $L$ admits a primitive embedding (unique up to isometries) into $\Lambda_{K3}$ whose orthogonal complement splits off two copies of $U$, this follows from the fact that there is a primitive embedding of $A_1^{\oplus 3}$ into $E_8^{\oplus 2}$. Since the discriminant group $\left(A_L,q_L\right)$ of $L$ and $\left(A_{L^{\perp}},q_{L^\perp}\right)$ of $L^{\perp}$ are isomorphic up to sign, one computes that the discriminant, level, and associated weight for $L^{\perp}$ is given by $\left(D,N,k\right)=\left(8,4,\frac{17}{2}\right)$. Then 
\[
\frac{(2\pi)^{\frac{17}{2}}}{\sqrt{8}\cdot\Gamma\left(\frac{17}{2}\right)\cdot\zeta(8)}=152.7645688\ldots\;\;\;\hbox{and}\;\;\; C(N,k)=\prod_{p\mid 8}\frac{p^2-1}{p^2(1-p^{-16})}=\frac{16384}{21845}. 
\]
In particular
\[
60=4\cdot \dim\left(\mathcal{M}_L\right)<\frac{(2\pi)^{\frac{17}{2}}}{\sqrt{8}\cdot\Gamma\left(\frac{17}{2}\right)\cdot\zeta(8)}\cdot C(N,k).
\]
In examples $(1)-(11), (13), (14), (16), (17), (18), (46)$, and $(49)$, the inequality \eqref{eq:sec2:main} is not satisfied, yet the inequality of Proposition \ref{sec3:prop:uniruled} still holds for $(1), (9), (13), (16), (18),$ and $(49)$. All these are obtained by computing $c_{1,0}(E_{k,L^\perp})$ explicitly\footnote{The computation of Fourier coefficients of vector-valued Eisenstein series $E_{k,L}$ has been implemented in \cite{weilrep}.}, see Lemma \ref{lemma:sec2:(1,0)-coef}. We illustrate the argument by treating example (1). In this case \cite{Rou22a}*{Section 3.4} the Gram matrix of the Ner\'on--Severi lattice is
\[
L={\rm{NS}}\left(X\right)\cong S_4=\left(\begin{array}{ccc}2&1&2\\1&-2&1\\2&1&-2\end{array}\right),
\]
it has signature $(1,2)$, and via the identification $\left(A_L,-q_L\right)\cong\left(A_{L^\perp},q_{L^\perp}\right)$ one has that $A_{L^\perp}\cong\mathbb{Z}\big/20\mathbb{Z}$ with generator of square $-\frac{3}{20}$. Recall that the Weil representation only depends on the quadratic module $\left(A_{L^\perp},q_{L^\perp}\right)$. One can take as such the discriminant group of $U\oplus S_4(-1)$. In this case one computes
\[
c_{1,0}\left(E_{\frac{19}{2},L^{\perp}}\right)=-\frac{5912665925814}{82295676409}=-71.846\ldots
\]
and by Proposition \ref{sec3:prop:uniruled} one concludes uniruledness.
\end{proof}


\begin{table}[hbt!]
\caption{List of $55$ moduli spaces parameterizing lattice-polarized K3 surfaces with $\rho\geq 3$ whose general element has a fixed finite automorphism group where no unirational parametrization was known. The last column lists those cases where our methods give us uniruledness. We follow the notation in \cite{Rou22}. The rationality of the last two examples was recently established in \cite{CMW24}.}
\label{sec4:table:thm4.1}
\begin{tabular}{|c|c|c|c|c|c|}
\hline
&\hbox{${\rm{NS}}(X)$}&$(D,N,k)$ of $L^\perp$&${\rm{Aut}}(X)$&$\dim\left(\mathcal{M}_L\right)$ &\hbox{Uniruled}\\ \hline
$1$&$S_4$&$\left(20,40,\frac{19}{2}\right)$&$\mathbb{Z}\big/2\mathbb{Z}$&$17$&Prop. \ref{sec3:prop:uniruled} \\ \hline
$2$&$S_{1,1,6}$&$\left(72,36,\frac{19}{2}\right)$&$\mathbb{Z}\big/2\mathbb{Z}$&$17$&? \\ \hline
$3$&$S_{1,1,8}$&$\left(128,64,\frac{19}{2}\right)$&$\{1\}$&$17$&? \\ \hline
$4$&$S_{1,9,1}$&$\left(162,108,\frac{19}{2}\right)$&$\{1\}$&$17$&? \\ \hline
$5$&$S_{7,1,1}$&$\left(98,14,\frac{19}{2}\right)$&$\{1\}$&$17$&? \\ \hline
$6$&$S_{10,1,1}$&$\left(200,20,\frac{19}{2}\right)$&$\mathbb{Z}\big/2\mathbb{Z}$&$17$&? \\ \hline
$7$&$S_{12,1,1}$&$\left(288,12,\frac{19}{2}\right)$&$\{1\}$&$17$&? \\ \hline
$8$&$S_{4,1,2}'$&$\left(32,8,\frac{19}{2}\right)$&$\mathbb{Z}\big/2\mathbb{Z}$&$17$&? \\ \hline
$9$&$L(24)$&$\left(28,14,9\right)$&$\mathbb{Z}\big/2\mathbb{Z}$&$16$&Prop. \ref{sec3:prop:uniruled} \\ \hline
$10$&$L(27)$&$\left(60,30,9\right)$&$\mathbb{Z}\big/2\mathbb{Z}$&$16$&? \\ \hline
$11$&$[4]\oplus [-4]\oplus A_2(-1)$&$\left(48,24,9\right)$&$\mathbb{Z}\big/2\mathbb{Z}$&$16$&? \\ \hline
$12$&$U\oplus A_1(-1)^{\oplus 3}$&$\left(8,4,\frac{17}{2}\right)$&$\mathbb{Z}\big/2\mathbb{Z}$&$15$&Thm. \ref{thm:sec2:main} \\ \hline
$13$&$U(2)\oplus A_1(-1)^{\oplus 3}$&$\left(32,4,\frac{17}{2}\right)$&$\mathbb{Z}\big/2\mathbb{Z}$&$15$&Prop. \ref{sec3:prop:uniruled} \\ \hline
$14$&$U(4)\oplus A_1(-1)^{\oplus 3}$&$\left(128,4,\frac{17}{2}\right)$&$\mathbb{Z}\big/2\mathbb{Z}$&$15$&? \\ \hline
$15$&$[4]\oplus D_4(-1)$&$\left(16,8,\frac{17}{2}\right)$&$\mathbb{Z}\big/2\mathbb{Z}$&$15$&Thm. \ref{thm:sec2:main} \\ \hline
$16$&$[8]\oplus D_4(-1)$&$\left(32,16,\frac{17}{2}\right)$&$\mathbb{Z}\big/2\mathbb{Z}$&$15$&Prop. \ref{sec3:prop:uniruled} \\ \hline
$17$&$[16]\oplus D_4(-1)$&$\left(64,32,\frac{17}{2}\right)$&$\mathbb{Z}\big/2\mathbb{Z}$&$15$&? \\ \hline
$18$&$U(4)\oplus D_4(-1)$&$\left(64,4,8\right)$&$\mathbb{Z}\big/2\mathbb{Z}$&$14$&Prop. \ref{sec3:prop:uniruled} \\ \hline
$19$&$U\oplus A_4(-1)$&$\left(5,5,8\right)$&$\mathbb{Z}\big/2\mathbb{Z}$&$14$&Thm. \ref{thm:sec2:main} \\ \hline
$20$&$U\oplus A_1(-1)\oplus A_3(-1)$&$\left(8,8,8\right)$&$\mathbb{Z}\big/2\mathbb{Z}$&$14$&Thm. \ref{thm:sec2:main} \\ \hline
$21$&$U\oplus A_2(-1)^{\oplus 2}$&$\left(9,3,8\right)$&$\mathbb{Z}\big/2\mathbb{Z}$&$14$&Thm. \ref{thm:sec2:main} \\ \hline
$22$&$U\oplus A_1(-1)^{\oplus 2}\oplus A_2(-1)$&$\left(12,12,8\right)$&$\mathbb{Z}\big/2\mathbb{Z}$&$14$&Thm. \ref{thm:sec2:main} \\ \hline
$23$&$U\oplus A_1(-1)^{\oplus 4}$&$\left(16,4,8\right)$&$\mathbb{Z}\big/2\mathbb{Z}$&$14$&Thm. \ref{thm:sec2:main} \\ \hline
$24$&$U\oplus D_4(-1)\oplus A_1(-1)$&$\left(8,4,\frac{15}{2}\right)$&$\mathbb{Z}\big/2\mathbb{Z}$&$13$&Thm. \ref{thm:sec2:main} \\ \hline
$25$&$U\oplus A_1(-1)\oplus A_2(-1)^{\oplus 2}$&$\left(18,12,\frac{15}{2}\right)$&$\mathbb{Z}\big/2\mathbb{Z}$&$13$&Thm. \ref{thm:sec2:main} \\ \hline
$26$&$U\oplus A_1(-1)^{\oplus 2}\oplus A_3(-1)$&$\left(16,8,\frac{15}{2}\right)$&$\mathbb{Z}\big/2\mathbb{Z}$&$13$&Thm. \ref{thm:sec2:main} \\ \hline
$27$&$U\oplus A_2(-1)\oplus A_3(-1)$&$\left(12,24,\frac{15}{2}\right)$&$\mathbb{Z}\big/2\mathbb{Z}$&$13$&Thm. \ref{thm:sec2:main} \\ \hline
$28$&$U\oplus A_1(-1)\oplus A_4(-1)$&$\left(10,20,\frac{15}{2}\right)$&$\mathbb{Z}\big/2\mathbb{Z}$&$13$&Thm. \ref{thm:sec2:main} \\ \hline
$29$&$U\oplus A_5(-1)$&$\left(6,12,\frac{15}{2}\right)$&$\mathbb{Z}\big/2\mathbb{Z}$&$13$&Thm. \ref{thm:sec2:main} \\ \hline
$30$&$U\oplus D_6(-1)$&$\left(4,4,7\right)$&$\mathbb{Z}\big/2\mathbb{Z}$&$12$&Thm. \ref{thm:sec2:main} \\ \hline
$31$&$U\oplus D_4(-1)\oplus A_1(-1)^{\oplus 2}$&$\left(16,4,7\right)$&$\mathbb{Z}\big/2\mathbb{Z}$&$12$&Thm. \ref{thm:sec2:main} \\ \hline
$32$&$U\oplus A_2(-1)^{\oplus 3}$&$\left(27,8,7\right)$&$\mathbb{Z}\big/2\mathbb{Z}$&$12$&Thm. \ref{thm:sec2:main} \\ \hline
$33$&$U\oplus A_3(-1)^{\oplus 2}$&$\left(16,8,7\right)$&$\mathbb{Z}\big/2\mathbb{Z}$&$12$&Thm. \ref{thm:sec2:main} \\ \hline
$34$&$U\oplus A_2(-1)\oplus A_4(-1)$&$\left(15,30,7\right)$&$\mathbb{Z}\big/2\mathbb{Z}$&$12$&Thm. \ref{thm:sec2:main} \\ \hline
$35$&$U\oplus A_1(-1)\oplus A_5(-1)$&$\left(12,12,7\right)$&$\mathbb{Z}\big/2\mathbb{Z}$&$12$&Thm. \ref{thm:sec2:main} \\ \hline
$36$&$U\oplus A_6(-1)$&$\left(7,7,7\right)$&$\mathbb{Z}\big/2\mathbb{Z}$&$12$&Thm. \ref{thm:sec2:main} \\ \hline
$37$&$U\oplus D_5(-1)\oplus A_1(-1)$&$\left(8,4,7\right)$&$\mathbb{Z}\big/2\mathbb{Z}$&$12$&Thm. \ref{thm:sec2:main} \\ \hline
$38$&$U\oplus D_6(-1)\oplus A_1(-1)$&$\left(8,4,\frac{13}{2}\right)$&$\mathbb{Z}\big/2\mathbb{Z}$&$11$&Thm. \ref{thm:sec2:main} \\ \hline
$39$&$U\oplus D_4(-1)\oplus A_1(-1)^{\oplus 3}$&$\left(32,4,\frac{13}{2}\right)$&$\mathbb{Z}\big/2\mathbb{Z}$&$11$&Thm. \ref{thm:sec2:main} \\ \hline
$40$&$U\oplus A_7(-1)$&$\left(8,16,\frac{13}{2}\right)$&$\mathbb{Z}\big/2\mathbb{Z}$&$11$&Thm. \ref{thm:sec2:main} \\ \hline
$41$&$U\oplus D_4(-1)\oplus A_3(-1)$&$\left(16,8,\frac{13}{2}\right)$&$\mathbb{Z}\big/2\mathbb{Z}$&$11$&Thm. \ref{thm:sec2:main} \\ \hline
$42$&$U\oplus D_5(-1)\oplus A_2(-1)$&$\left(12,24,\frac{13}{2}\right)$&$\mathbb{Z}\big/2\mathbb{Z}$&$11$& Thm. \ref{thm:sec2:main} \\ \hline
$43$&$U\oplus D_7(-1)$&$\left(4,8,\frac{13}{2}\right)$&$\mathbb{Z}\big/2\mathbb{Z}$&$11$& Thm. \ref{thm:sec2:main} \\ \hline
\end{tabular}
\end{table}

\begin{table}[hbt!]
\begin{tabular}{|c|c|c|c|c|c|}
\hline
&\hbox{${\rm{NS}}(X)$}&$(D,N,k)$ of $L^\perp$&${\rm{Aut}}(X)$&$\dim\left(\mathcal{M}_L\right)$ &\hbox{Uniruled}\\ \hline
$44$&$U\oplus D_8(-1)$&$\left(4,2,6\right)$&$\mathbb{Z}\big/2\mathbb{Z}$&$10$&Thm. \ref{thm:sec2:main} \\ \hline
$45$&$U\oplus D_6(-1)\oplus A_1(-1)^{\oplus 2}$&$\left(16,4,6\right)$&$\mathbb{Z}\big/2\mathbb{Z}$&$10$&Thm. \ref{thm:sec2:main} \\ \hline
$46$&$U\oplus A_1(-1)^{\oplus 8}$&$\left(256,4,6\right)$&$\left(\mathbb{Z}\big/2\mathbb{Z}\right)^2$&$10$&? \\ \hline
$47$&$U\oplus D_8(-1)\oplus A_1(-1)$&$\left(8,4,\frac{11}{2}\right)$&$\mathbb{Z}\big/2\mathbb{Z}$&$9$&Thm. \ref{thm:sec2:main} \\ \hline
$48$&$U\oplus D_4(-1)^{\oplus 2}\oplus A_1(-1)$&$\left(32,4,\frac{11}{2}\right)$&$\mathbb{Z}\big/2\mathbb{Z}$&$9$&Thm. \ref{thm:sec2:main} \\ \hline
$49$&$U\oplus D_4(-1)\oplus A_1(-1)^{\oplus 5}$&$\left(128,4,\frac{11}{2}\right)$&$\left(\mathbb{Z}\big/2\mathbb{Z}\right)^2$&$9$&Prop. \ref{sec3:prop:uniruled} \\ \hline
$50$&$U\oplus E_8(-1)\oplus A_1(-1)^{\oplus 2}$&$\left(4,4,5\right)$&$\mathbb{Z}\big/2\mathbb{Z}$&$8$&Thm. \ref{thm:sec2:main} \\ \hline
$51$&$U\oplus D_8(-1)\oplus A_1(-1)^{\oplus 2}$&$\left(16,4,5\right)$&$\mathbb{Z}\big/2\mathbb{Z}$&$8$&Thm. \ref{thm:sec2:main} \\ \hline
$52$&$U\oplus D_4(-1)^{\oplus 2}\oplus A_1(-1)^{\oplus 2}$&$\left(64,4,5\right)$&$\left(\mathbb{Z}\big/2\mathbb{Z}\right)^2$&$8$&Thm. \ref{thm:sec2:main}\\ \hline
$53$&$U\oplus E_8(-1)\oplus A_3(-1)$&$\left(4,8,\frac{9}{2}\right)$&$\mathbb{Z}\big/2\mathbb{Z}$&$7$&Thm. \ref{thm:sec2:main} \\ \hline
$54$&$U\oplus E_8(-1)\oplus A_1(-1)^{\oplus 4}$&$\left(16,4,4\right)$&$\left(\mathbb{Z}\big/2\mathbb{Z}\right)^2$&$6$&\cite{CMW24}, Thm. \ref{thm:sec2:main} \\ \hline
$55$&$U\oplus E_8(-1)\oplus D_4(-1)\oplus A_1(-1)$&$\left(8,4,\frac{7}{2}\right)$&$\left(\mathbb{Z}\big/2\mathbb{Z}\right)^2$&$5$& \cite{CMW24}, Thm. \ref{thm:sec2:main} \\ \hline
\end{tabular}
\end{table}

\subsection{Moduli of Kummer surfaces}

Let $\mathcal{A}_d$ be the moduli of $(1,d)$-polarized abelian surfaces. It is known that $\mathcal{A}_d$ is unirational for $1\leq d\leq12$ and $d=14,16,18,20$, see \cites{Igu72, HM73, LB89, OGr89, GP98, MS01, GP01, GP01b, GP11}. Further, $\mathcal{A}_d$ is not unirational for $d\geq37$, $d=17, 19, 21, 22, 23$, and some other values in the range $24\leq d\leq 36$, see \cite{Gri95}*{Theorem 1}. Let $L_d=U^{\oplus 2}\oplus A_1(-d)$. The moduli space $\mathcal{A}_d$ admits a finite map
\begin{equation}
\label{sec4:eq:AdAd*}
\eta:\mathcal{A}_d\longrightarrow \mathcal{A}_d^*:=\mathcal{D}_{L_d}\big/{\rm{O}}^+\left(L_d\right)
\end{equation}
that factors through $\mathcal{F}_{L_d}$. Further, the {\it{minimal Siegel modular threefold}} $\mathcal{A}_{d}^*$ can be seen as the moduli space of Kummer surfaces \cite{GH98}*{Theorem 1.5}. The unirationality of $\mathcal{A}_d$ implies unirationality (and therefore unirulednes) of $\mathcal{A}_d^*$. In \cite{GH14}*{Theorem 3.1} the uniruledness of $\mathcal{A}_{21}^*$ is obtained by constructing a reflective modular form with respect to ${\rm{O}}^+\left(L_{21}\right)$, cf. \cite{GN02}*{Theorem 2.2.3}. Note that in this case $\mathcal{A}_{21}$ has non-negative Kodaira dimension, see \cites{Gri94, Gri95}. We obtain one more case where $\mathcal{A}_{d}^*$ has negative Kodaira dimension.

\begin{proposition}
\label{sec4:prop:F13}
The modular variety $\mathcal{F}_{L_{13}}$ is uniruled. 
\end{proposition}

\begin{proof}
From Lemma \ref{lemma:sec2:(1,0)-coef} one obtains (see also \cite{weilrep}): 
\[
c_{1,0}\left(E_{\frac{5}{2},L_{13}}\right)=-\frac{264}{17}=-15.52941176.
\]
Uniruledness follows from Proposition \ref{sec3:prop:uniruled}.
\end{proof}

\begin{corollary}
The moduli space of Kummer surfaces $\mathcal{A}_d^*$ associated to $(1,d)$-polarized abelian surfaces is uniruled for $d=13$.
\end{corollary}

\begin{proof}
The corollary follows from Proposition \ref{sec4:prop:F13} and the fact that the map \eqref{sec4:eq:AdAd*} factors through $\mathcal{F}_{L_d}$.
\end{proof}

\begin{remark}
Further uniruledness results can be obtained for moduli spaces parameterizing special families of cubic fourfolds or hyperk\"{a}hler manifolds. This can also be extended to moduli spaces of irreducible symplectic varieties (singular), see \cite{BL22}. 
\end{remark}

\bibliography{Bibliography}

\begin{bibdiv}
\begin{biblist}

\bib{AMRT10}{book}{
      author={Ash, A.},
      author={Mumford, D.},
      author={Rapoport, M.},
      author={Tai, Y.-S.},
       title={Smooth compactifications of locally symmetric varieties},
     edition={Second},
      series={Cambridge Mathematical Library},
   publisher={Cambridge University Press, Cambridge},
        date={2010},
        ISBN={978-0-521-73955-9},
         url={https://doi.org/10.1017/CBO9780511674693},
        note={With the collaboration of Peter Scholze},
}

\bib{BB66}{article}{
      author={Baily, W.~L., Jr.},
      author={Borel, A.},
       title={Compactification of arithmetic quotients of bounded symmetric
  domains},
        date={1966},
        ISSN={0003-486X},
     journal={Ann. of Math. (2)},
      volume={84},
       pages={442\ndash 528},
         url={https://doi.org/10.2307/1970457},
}

\bib{BBBF23}{article}{
      author={Barros, I.},
      author={Beri, P.},
      author={Brakkee, E.},
      author={Flapan, L.},
       title={Kodaira dimension of moduli spaces of hyperk\"{a}hler varieties},
     journal={Algebr. Geom. (to appear)},
      eprint={https://arxiv.org/abs/2212.12586},
}

\bib{BBFW24}{misc}{
      author={Barros, I.},
      author={Beri, P.},
      author={Flapan, L.},
      author={Williams, B.},
       title={Cones of {N}oether-{L}efschetz divisors and moduli spaces of
  hyperk\"ahler manifolds},
        date={2024},
}

\bib{BDPP13}{article}{
      author={Boucksom, S.},
      author={Demailly, J.-P.},
      author={P\u{a}un, M.},
      author={Peternell, T.},
       title={The pseudo-effective cone of a compact {K}\"{a}hler manifold and
  varieties of negative {K}odaira dimension},
        date={2013},
        ISSN={1056-3911,1534-7486},
     journal={J. Algebraic Geom.},
      volume={22},
      number={2},
       pages={201\ndash 248},
         url={https://doi.org/10.1090/S1056-3911-2012-00574-8},
}

\bib{Bel02}{article}{
      author={Belcastro, S.-M.},
       title={Picard lattices of families of {K3} surfaces},
        date={2002},
        ISSN={0092-7872,1532-4125},
     journal={Comm. Algebra},
      volume={30},
      number={1},
       pages={61\ndash 82},
         url={https://doi.org/10.1081/AGB-120006479},
}

\bib{BK01}{article}{
      author={Bruinier, J.~H.},
      author={Kuss, M.},
       title={Eisenstein series attached to lattices and modular forms on
  orthogonal groups},
        date={2001},
        ISSN={0025-2611,1432-1785},
     journal={Manuscripta Math.},
      volume={106},
      number={4},
       pages={443\ndash 459},
         url={https://doi.org/10.1007/s229-001-8027-1},
}

\bib{BL22}{article}{
      author={Bakker, B.},
      author={Lehn, C.},
       title={The global moduli theory of symplectic varieties},
        date={2022},
        ISSN={0075-4102,1435-5345},
     journal={J. Reine Angew. Math.},
      volume={790},
       pages={223\ndash 265},
         url={https://doi.org/10.1515/crelle-2022-0033},
}

\bib{BM19}{article}{
      author={Bruinier, J.~H.},
      author={M\"{o}ller, M.},
       title={Cones of {H}eegner divisors},
        date={2019},
        ISSN={1056-3911},
     journal={J. Algebraic Geom.},
      volume={28},
      number={3},
       pages={497\ndash 517},
         url={https://doi.org/10.1090/jag/734},
}

\bib{Bor98}{article}{
      author={Borcherds, R.~E.},
       title={Automorphic forms with singularities on {G}rassmannians},
        date={1998},
        ISSN={0020-9910,1432-1297},
     journal={Invent. Math.},
      volume={132},
      number={3},
       pages={491\ndash 562},
         url={https://doi.org/10.1007/s002220050232},
}

\bib{Bor99}{article}{
      author={Borcherds, R.~E.},
       title={The {G}ross-{K}ohnen-{Z}agier theorem in higher dimensions},
        date={1999},
        ISSN={0012-7094,1547-7398},
     journal={Duke Math. J.},
      volume={97},
      number={2},
       pages={219\ndash 233},
         url={https://doi.org/10.1215/S0012-7094-99-09710-7},
}

\bib{Bru02}{book}{
      author={Bruinier, J.~H.},
       title={Borcherds products on {O}(2, {$l$}) and {C}hern classes of
  {H}eegner divisors},
      series={Lecture Notes in Mathematics},
   publisher={Springer-Verlag, Berlin},
        date={2002},
      volume={1780},
        ISBN={3-540-43320-1},
         url={https://doi.org/10.1007/b83278},
}

\bib{CMW24}{misc}{
      author={Clingher, A.},
      author={Malmendier, A.},
      author={Williams, B.},
       title={{K3} surfaces and orthogonal modular morms},
        date={2024},
         url={https://arxiv.org/abs/2411.05970},
}

\bib{Dol96}{article}{
      author={Dolgachev, I.~V.},
       title={Mirror symmetry for lattice polarized {K3} surfaces},
        date={1996},
        ISSN={1072-3374},
     journal={J. Math. Sci.},
      volume={81},
      number={3},
       pages={2599\ndash 2630},
         url={https://doi.org/10.1007/BF02362332},
}

\bib{Don84}{article}{
      author={Donagi, R.},
       title={The unirationality of {${\scr A}\sb{5}$}},
        date={1984},
        ISSN={0003-486X,1939-8980},
     journal={Ann. of Math. (2)},
      volume={119},
      number={2},
       pages={269\ndash 307},
         url={https://doi.org/10.2307/2007041},
}

\bib{DSMS21}{article}{
      author={Dittmann, M.},
      author={Salvati~Manni, R.},
      author={Scheithauer, N.~R.},
       title={Harmonic theta series and the {K}odaira dimension of
  {$\mathcal{A}_6$}},
        date={2021},
        ISSN={1937-0652,1944-7833},
     journal={Algebra Number Theory},
      volume={15},
      number={1},
       pages={271\ndash 285},
         url={https://doi.org/10.2140/ant.2021.15.271},
}

\bib{Fre83}{book}{
      author={Freitag, E.},
       title={Siegelsche {M}odulfunktionen},
      series={Grundlehren der mathematischen Wissenschaften},
   publisher={Springer-Verlag, Berlin},
        date={1983},
      volume={254},
        ISBN={3-540-11661-3},
         url={https://doi.org/10.1007/978-3-642-68649-8},
}

\bib{FV12}{article}{
      author={Farkas, G.},
      author={Verra, A.},
       title={Moduli of theta-characteristics via {N}ikulin surfaces},
        date={2012},
        ISSN={0025-5831,1432-1807},
     journal={Math. Ann.},
      volume={354},
      number={2},
       pages={465\ndash 496},
         url={https://doi.org/10.1007/s00208-011-0739-z},
}

\bib{FV16}{article}{
      author={Farkas, G.},
      author={Verra, A.},
       title={Prym varieties and moduli of polarized {N}ikulin surfaces},
        date={2016},
        ISSN={0001-8708,1090-2082},
     journal={Adv. Math.},
      volume={290},
       pages={314\ndash 328},
         url={https://doi.org/10.1016/j.aim.2015.12.006},
}

\bib{FV14}{article}{
      author={Farkas, G.},
      author={Verra, A.},
       title={The universal {K3} surface of genus 14 via cubic fourfolds},
        date={2018},
        ISSN={0021-7824,1776-3371},
     journal={J. Math. Pures Appl. (9)},
      volume={111},
       pages={1\ndash 20},
         url={https://doi.org/10.1016/j.matpur.2017.07.014},
}

\bib{FV21}{article}{
      author={Farkas, G.},
      author={Verra, A.},
       title={The unirationality of the moduli space of {K3} surfaces of genus
  22},
        date={2021},
        ISSN={0025-5831,1432-1807},
     journal={Math. Ann.},
      volume={380},
      number={3-4},
       pages={953\ndash 973},
         url={https://doi.org/10.1007/s00208-020-02036-y},
}

\bib{GH14}{article}{
      author={Gritsenko, V.},
      author={Hulek, K.},
       title={Uniruledness of orthogonal modular varieties},
        date={2014},
        ISSN={1056-3911,1534-7486},
     journal={J. Algebraic Geom.},
      volume={23},
      number={4},
       pages={711\ndash 725},
         url={https://doi.org/10.1090/S1056-3911-2014-00632-9},
}

\bib{GH98}{article}{
      author={Gritsenko, V.},
      author={Hulek, K.},
       title={Minimal {S}iegel modular threefolds},
        date={1998},
        ISSN={0305-0041,1469-8064},
     journal={Math. Proc. Cambridge Philos. Soc.},
      volume={123},
      number={3},
       pages={461\ndash 485},
         url={https://doi.org/10.1017/S0305004197002259},
}

\bib{GHS07b}{article}{
      author={Gritsenko, V.},
      author={Hulek, K.},
      author={Sankaran, G.~K.},
       title={The {H}irzebruch-{M}umford volume for the orthogonal group and
  applications},
        date={2007},
        ISSN={1431-0635,1431-0643},
     journal={Doc. Math.},
      volume={12},
       pages={215\ndash 241},
}

\bib{GHS07}{article}{
      author={Gritsenko, V.~A.},
      author={Hulek, K.},
      author={Sankaran, G.~K.},
       title={The {K}odaira dimension of the moduli of {K3} surfaces},
        date={2007},
        ISSN={0020-9910},
     journal={Invent. Math.},
      volume={169},
      number={3},
       pages={519\ndash 567},
         url={https://doi.org/10.1007/s00222-007-0054-1},
}

\bib{GHS10}{article}{
      author={Gritsenko, V.},
      author={Hulek, K.},
      author={Sankaran, G.~K.},
       title={Moduli spaces of irreducible symplectic manifolds},
        date={2010},
        ISSN={0010-437X,1570-5846},
     journal={Compos. Math.},
      volume={146},
      number={2},
       pages={404\ndash 434},
         url={https://doi.org/10.1112/S0010437X0900445X},
}

\bib{GHS11}{article}{
      author={Gritsenko, V.},
      author={Hulek, K.},
      author={Sankaran, G.~K.},
       title={Moduli spaces of polarized symplectic {O}'{G}rady varieties and
  {B}orcherds products},
        date={2011},
        ISSN={0022-040X,1945-743X},
     journal={J. Differential Geom.},
      volume={88},
      number={1},
       pages={61\ndash 85},
         url={http://projecteuclid.org/euclid.jdg/1317758869},
}

\bib{GN02}{article}{
      author={Gritsenko, V.~A.},
      author={Nikulin, V.~V.},
       title={On the classification of {L}orentzian {K}ac-{M}oody algebras},
        date={2002},
        ISSN={0042-1316,2305-2872},
     journal={Uspekhi Mat. Nauk},
      volume={57},
      number={5(347)},
       pages={79\ndash 138},
         url={https://doi.org/10.1070/RM2002v057n05ABEH000553},
}

\bib{GP01b}{article}{
      author={Gross, M.},
      author={Popescu, S.},
       title={Calabi--yau threefolds and moduli of abelian surfaces {I}},
        date={2001},
     journal={Compos. Math.},
      volume={127},
      number={2},
       pages={169\ndash 228},
         url={https://doi.org/10.1023/A:1012076503121},
}

\bib{GP01}{article}{
      author={Gross, M.},
      author={Popescu, S.},
       title={The moduli space of {$(1,11)$}-polarized abelian surfaces is
  unirational},
        date={2001},
     journal={Compos. Math.},
      volume={126},
      number={1},
       pages={1\ndash 24},
}

\bib{GP11}{article}{
      author={Gross, M.},
      author={Popescu, S.},
       title={Calabi-{Y}au three-folds and moduli of abelian surfaces {II}},
        date={2011},
        ISSN={0002-9947,1088-6850},
     journal={Trans. Amer. Math. Soc.},
      volume={363},
      number={7},
       pages={3573\ndash 3599},
         url={https://doi.org/10.1090/S0002-9947-2011-05179-2},
}

\bib{GP98}{article}{
      author={Gross, M.},
      author={Popescu, S.},
       title={Equations of {$(1,d)$}-polarized abelian surfaces},
        date={1998},
     journal={Math. Ann.},
      volume={310},
      number={2},
       pages={333\ndash 377},
         url={https://doi.org/10.1007/s002080050151},
}

\bib{Gri10}{misc}{
      author={Gritsenko, V.},
       title={Reflective modular forms in algebraic geometry},
        date={2010},
         url={https://arxiv.org/abs/1005.3753},
}

\bib{Gri94}{article}{
      author={Gritsenko, V.},
       title={Modular forms and moduli spaces of abelian and {K3} surfaces},
        date={1994},
        ISSN={0234-0852},
     journal={Algebra i Analiz},
      volume={6},
      number={6},
       pages={65\ndash 102},
}

\bib{Gri95}{incollection}{
      author={Gritsenko, V.},
       title={Irrationality of the moduli spaces of polarized abelian
  surfaces},
        date={1995},
   booktitle={Abelian varieties ({E}gloffstein, 1993)},
   publisher={de Gruyter, Berlin},
       pages={63\ndash 84},
        note={With an appendix by the author and K. Hulek},
}

\bib{HM73}{article}{
      author={Horrocks, G.},
      author={Mumford, D.},
       title={A rank $2$ vector bundle on $\mathbb{P}^4$ with $15,000$
  symmetries},
        date={1973},
        ISSN={0040-9383},
     journal={Topology},
      volume={12},
      number={1},
       pages={63\ndash 81},
}

\bib{Igu72}{book}{
      author={Igusa, J.},
       title={Theta functions},
      series={Grundlehren der mathematischen Wissenschaften},
   publisher={Springer-Verlag, New York-Heidelberg},
        date={1972},
      volume={194},
}

\bib{KLV20}{article}{
      author={Knutsen, A.~L.},
      author={Lelli-Chiesa, M.},
      author={Verra, A.},
       title={Moduli of non-standard {N}ikulin surfaces in low genus},
        date={2020},
        ISSN={0391-173X,2036-2145},
     journal={Ann. Sc. Norm. Super. Pisa Cl. Sci. (5)},
      volume={21},
       pages={361\ndash 384},
}

\bib{KLV21}{article}{
      author={Knutsen, A.~L.},
      author={Lelli-Chiesa, M.},
      author={Verra, A.},
       title={Half {N}ikulin surfaces and moduli of {P}rym curves},
        date={2021},
        ISSN={1474-7480,1475-3030},
     journal={J. Inst. Math. Jussieu},
      volume={20},
      number={5},
       pages={1547\ndash 1584},
         url={https://doi.org/10.1017/S1474748019000574},
}

\bib{KM90}{article}{
      author={Kudla, S.~S.},
      author={Millson, J.~J.},
       title={Intersection numbers of cycles on locally symmetric spaces and
  {F}ourier coefficients of holomorphic modular forms in several complex
  variables},
        date={1990},
        ISSN={1618-1913},
     journal={Publ. Math. Inst. Hautes \'{E}tudes Sci.},
      volume={71},
       pages={121\ndash 172},
         url={https://doi.org/10.1007/BF02699880},
}

\bib{Kon89}{article}{
      author={Kond\={o}, S.},
       title={Algebraic {K3} surfaces with finite automorphism groups},
        date={1989},
        ISSN={0027-7630,2152-6842},
     journal={Nagoya Math. J.},
      volume={116},
       pages={1\ndash 15},
         url={https://doi.org/10.1017/S0027763000001653},
}

\bib{Kon93}{article}{
      author={Kond\={o}, S.},
       title={On the {K}odaira dimension of the moduli space of {K3} surfaces},
        date={1993},
        ISSN={0010-437X,1570-5846},
     journal={Compos. Math.},
      volume={89},
      number={3},
       pages={251\ndash 299},
         url={http://www.numdam.org/item?id=CM_1993__89_3_251_0},
}

\bib{Kon99}{article}{
      author={Kond\={o}, S.},
       title={On the {K}odaira dimension of the moduli space of {K3} surfaces.
  {II}},
        date={1999},
        ISSN={0010-437X,1570-5846},
     journal={Compos. Math.},
      volume={116},
      number={2},
       pages={111\ndash 117},
         url={https://doi.org/10.1023/A:1000675831026},
}

\bib{Kud03}{article}{
      author={Kudla, S.~S.},
       title={Integrals of {B}orcherds forms},
        date={2003},
        ISSN={0010-437X,1570-5846},
     journal={Compos. Math.},
      volume={137},
      number={3},
       pages={293\ndash 349},
         url={https://doi.org/10.1023/A:1024127100993},
}

\bib{LB89}{article}{
      author={Lange, H.},
      author={Birkenhake, C.},
       title={Abelian surfaces of type {$(1, 4)$}},
        date={1989},
     journal={Math. Ann.},
      volume={285},
      number={4},
       pages={625\ndash 646},
}

\bib{Ma18}{article}{
      author={Ma, S.},
       title={On the {K}odaira dimension of orthogonal modular varieties},
        date={2018},
        ISSN={0020-9910,1432-1297},
     journal={Invent. Math.},
      volume={212},
      number={3},
       pages={859\ndash 911},
         url={https://doi.org/10.1007/s00222-017-0781-x},
}

\bib{McG03}{article}{
      author={McGraw, W.~J.},
       title={The rationality of vector valued modular forms associated with
  the {W}eil representation},
        date={2003},
        ISSN={0025-5831,1432-1807},
     journal={Math. Ann.},
      volume={326},
      number={1},
       pages={105\ndash 122},
         url={https://doi.org/10.1007/s00208-003-0413-1},
}

\bib{MM86}{article}{
      author={Miyaoka, Y.},
      author={Mori, S.},
       title={A numerical criterion for uniruledness},
        date={1986},
        ISSN={0003-486X,1939-8980},
     journal={Ann. of Math. (2)},
      volume={124},
      number={1},
       pages={65\ndash 69},
         url={https://doi.org/10.2307/1971387},
}

\bib{MS01}{article}{
      author={Manolache, N.},
      author={Schreyer, F.-O.},
       title={Moduli of {$(1,7)$}-polarized abelian surfaces via syzygies},
        date={2001},
        ISSN={0025-584X,1522-2616},
     journal={Math. Nachr.},
      volume={226},
       pages={177\ndash 203},
}

\bib{Muk06}{incollection}{
      author={Mukai, S.},
       title={Polarized {K3} surfaces of genus thirteen},
        date={2006},
   booktitle={Moduli spaces and arithmetic geometry},
      series={Adv. Stud. Pure Math.},
      volume={45},
   publisher={Math. Soc. Japan, Tokyo},
       pages={315\ndash 326},
         url={https://doi.org/10.2969/aspm/04510315},
}

\bib{Muk10}{article}{
      author={Mukai, S.},
       title={Curves and symmetric spaces, {II}},
        date={2010},
        ISSN={0003-486X,1939-8980},
     journal={Ann. of Math. (2)},
      volume={172},
      number={3},
       pages={1539\ndash 1558},
         url={https://doi.org/10.4007/annals.2010.172.1539},
}

\bib{Muk88}{incollection}{
      author={Mukai, S.},
       title={Curves, {K3} surfaces and {F}ano {$3$}-folds of genus {$\leq
  10$}},
        date={1988},
   booktitle={Algebraic geometry and commutative algebra, {V}ol.\ {I}},
   publisher={Kinokuniya, Tokyo},
       pages={357\ndash 377},
}

\bib{Muk92}{incollection}{
      author={Mukai, S.},
       title={Polarized {K3} surfaces of genus {$18$} and {$20$}},
        date={1992},
   booktitle={Complex projective geometry ({T}rieste, 1989/{B}ergen, 1989)},
      series={London Math. Soc. Lecture Note Ser.},
      volume={179},
   publisher={Cambridge Univ. Press, Cambridge},
       pages={264\ndash 276},
         url={https://doi.org/10.1017/CBO9780511662652.019},
}

\bib{Muk96}{incollection}{
      author={Mukai, S.},
       title={Curves and {K3} surfaces of genus eleven},
        date={1996},
   booktitle={Moduli of vector bundles ({S}anda, 1994; {K}yoto, 1994)},
      series={Lecture Notes in Pure and Appl. Math.},
      volume={179},
   publisher={Dekker, New York},
       pages={189\ndash 197},
}

\bib{Mum83}{incollection}{
      author={Mumford, D.},
       title={On the {K}odaira dimension of the {S}iegel modular variety},
        date={1983},
   booktitle={Algebraic geometry---open problems ({R}avello, 1982)},
      series={Lecture Notes in Math.},
      volume={997},
   publisher={Springer, Berlin},
       pages={348\ndash 375},
         url={https://doi.org/10.1007/BFb0061652},
}

\bib{Nik79}{article}{
      author={Nikulin, V.~V.},
       title={Finite groups of automorphisms of {K}\"ahlerian {K3}\ surfaces},
        date={1979},
        ISSN={0134-8663},
     journal={Trudy Moskov. Mat. Obshch.},
      volume={38},
       pages={75\ndash 137},
}

\bib{Nik81}{incollection}{
      author={Nikulin, V.~V.},
       title={Quotient-groups of groups of automorphisms of hyperbolic forms by
  subgroups generated by {$2$}-reflections. {A}lgebro-geometric applications},
        date={1981},
   booktitle={Current problems in mathematics},
      series={Itogi Nauki i Tekhniki},
      volume={18},
   publisher={Akad. Nauk SSSR, Vsesoyuz. Inst. Nauchn. i Tekhn. Inform.,
  Moscow},
       pages={3\ndash 114},
}

\bib{Nik84}{incollection}{
      author={Nikulin, V.~V.},
       title={{K3}\ surfaces with a finite group of automorphisms and a
  {P}icard group of rank three},
        date={1984},
      volume={165},
       pages={119\ndash 142},
}

\bib{OGr89}{article}{
      author={O'Grady, K.~G.},
       title={On the {K}odaira dimension of moduli spaces of abelian surfaces},
        date={1989},
        ISSN={0010-437X},
     journal={Compos. Math.},
      volume={72},
       pages={121\ndash 163},
         url={http://www.numdam.org/item?id=CM_1989__72_2_121_0},
}

\bib{OS19}{article}{
      author={Opitz, S.},
      author={Schwagenscheidt, M.},
       title={Holomorphic {B}orcherds products of singular weight for simple
  lattices of arbitrary level},
        date={2019},
        ISSN={0002-9939,1088-6826},
     journal={Proc. Amer. Math. Soc.},
      volume={147},
      number={11},
       pages={4639\ndash 4653},
         url={https://doi.org/10.1090/proc/14650},
}

\bib{Pet15}{thesis}{
      author={Peterson, A.},
       title={Modular forms on the moduli space of polarised k3 surfaces},
        type={Ph.D. Thesis},
        date={2015},
}

\bib{Rou22a}{article}{
      author={Roulleau, X.},
       title={On the geometry of {K3} surfaces with finite automorphism group
  and no elliptic fibrations},
        date={2022},
        ISSN={0129-167X,1793-6519},
     journal={Internat. J. Math.},
      volume={33},
      number={6},
       pages={Paper No. 2250040, 39},
         url={https://doi.org/10.1142/S0129167X22500409},
}

\bib{Rou22}{article}{
      author={Roulleau, Xavier},
       title={An atlas of {K}3 surfaces with finite automorphism group},
        date={2022},
        ISSN={2491-6765},
     journal={\'Epijournal G\'eom. Alg\'ebrique},
      volume={6},
       pages={Art. 19, 95},
}

\bib{Tai82}{article}{
      author={Tai, Y.-S.},
       title={On the {K}odaira dimension of the moduli space of abelian
  varieties},
        date={1982},
        ISSN={0020-9910,1432-1297},
     journal={Invent. Math.},
      volume={68},
      number={3},
       pages={425\ndash 439},
         url={https://doi.org/10.1007/BF01389411},
}

\bib{TVA19}{article}{
      author={Tanimoto, S.},
      author={V\'arilly-Alvarado, A.},
       title={Kodaira dimension of moduli of special cubic fourfolds},
        date={2019},
        ISSN={0075-4102,1435-5345},
     journal={J. Reine Angew. Math.},
      volume={752},
       pages={265\ndash 300},
         url={https://doi.org/10.1515/crelle-2016-0053},
}

\bib{Ver16}{incollection}{
      author={Verra, A.},
       title={Geometry of genus {$8$} {N}ikulin surfaces and rationality of
  their moduli},
        date={2016},
   booktitle={K3 surfaces and their moduli},
      series={Progr. Math.},
      volume={315},
   publisher={Birkh\"auser/Springer, [Cham]},
       pages={345\ndash 364},
         url={https://doi.org/10.1007/978-3-319-29959-4_13},
}

\bib{vGS07}{article}{
      author={van Geemen, B.},
      author={Sarti, A.},
       title={Nikulin involutions on {K3} surfaces},
        date={2007},
        ISSN={0025-5874,1432-1823},
     journal={Math. Z.},
      volume={255},
      number={4},
       pages={731\ndash 753},
         url={https://doi.org/10.1007/s00209-006-0047-6},
}

\bib{Vin07}{article}{
      author={Vinberg, E.},
       title={Classification of {$2$}-reflective hyperbolic lattices of rank
  {$4$}},
        date={2007},
        ISSN={0134-8663},
     journal={Tr. Mosk. Mat. Obs.},
      volume={68},
       pages={44\ndash 76},
         url={https://doi.org/10.1090/s0077-1554-07-00160-4},
}

\bib{Wan24}{article}{
      author={Wang, H.},
       title={The classification of {$2$}-reflective modular forms},
        date={2024},
        ISSN={1435-9855,1435-9863},
     journal={J. Eur. Math. Soc. (JEMS)},
      volume={26},
      number={1},
       pages={111\ndash 151},
         url={https://doi.org/10.4171/jems/1358},
}

\bib{weilrep}{misc}{
      author={Williams, B.},
       title={weilrep},
         url={https://github.com/btw-47/weilrep},
        note={Available at \url{https://github.com/btw-47/weilrep}},
}

\bib{WW24}{article}{
      author={Wang, H.},
      author={Williams, B.},
       title={Modular forms with poles on hyperplane arrangements},
        date={2024},
        ISSN={2313-1691,2214-2584},
     journal={Algebr. Geom.},
      volume={11},
      number={4},
       pages={506\ndash 568},
}

\bib{Yon90}{article}{
      author={Yonemura, T.},
       title={Hypersurface simple {K3} singularities},
        date={1990},
        ISSN={0040-8735,2186-585X},
     journal={Tohoku Math. J. (2)},
      volume={42},
      number={3},
       pages={351\ndash 380},
         url={https://doi.org/10.2748/tmj/1178227616},
}

\end{biblist}
\end{bibdiv}
\bibliographystyle{alpha}
\end{document}